\documentclass[11pt]{amsart}
\usepackage{txfonts}
\usepackage{}
\usepackage{amssymb}
\theoremstyle{plain}
\newtheorem{Thm}{Theorem}

\begin{document}

\title[why cigar]
{Why cigar can not be isometrically immersed into the 3-space }

\author{Li Ma, Anqiang Zhu}
\address{Department of mathematics \\
Henan Normal university \\
Xinxiang, 453007 \\
China} \email{lma@tsinghua.edu.cn}

\address{Department of mathematics \\
Wuhan university \\
Wuhan, 430072 \\
China} \email{aqzhu.math@whu.edu.cn}

\thanks{The research is partially supported by the National Natural Science
Foundation of China (N0.11271111)}

\begin{abstract}
In this paper, we study the question if there is an isometric immersion of the cigar soliton into $R^3$. We show that the answer is negative.
Similar result in higher dimensions is also true for steady Bryant solitons.


\end{abstract}

 \maketitle

\section{Introduction}
  We consider in this short note the problem if there is a nontrivial n-dimensional steady Ricci soliton which can be embedded as a hypersurface in $R^{n+1}$. Recall that steady Ricci solitons are special solutions to Ricci flow introduced by R.Hamilton \cite{H} \cite{Cho06}.
  In dimension two, the only non-trivial complete Ricci soliton is the cigar. We show that it is impossible to realize it as a surface in $R^3$. The key step in proving it is to use the deep result of H.Wu \cite{W} about the convex surfaces. Similar result is also true for steady Bryant solitons higher dimensions. We believe that similar result is also true for radially symmetric expanding Ricci solitons. We shall use the notation $u=0(r)$ to denote by $C^{-1}r\leq u\leq Cr$ for some uniform constant $C>0$ and the uniform constant $C$ may vary from line to line.

\section{Cigar can not be immersed into $R^3$}\label{sect1}

Recall that the cigar soliton is a two dimensional Riemannian manifold $(R^2, g_\Sigma)$ with the Riemannian metric \cite{Cho06}\cite{H}
$$
g_\Sigma=\frac{dx^2+dy^2}{1+x^2+y^2}=\frac{dr^2+r^2d\theta^2}{1+r^2}.
$$
It has positive Gauss curvature
$$
K=\frac{2}{1+r^2}.
$$

If the cigar can be isometrically immersed into $R^3$. According to Theorem of Sacksteder-Heijenoort and the main theorem ($\delta$) of H.Wu in \cite{W}, it is the graph of a nonnegative strictly convex function $u$ defined in the plane $\{x_3=0\}$.

Recall that the induced metric of the graph of the function $z=u(x_1,x_2)$ is given by
$$
g=(\delta_{ij}+u_iu_j)dx^idx^j=g_{ij}dx^idx^j
$$
with its second fundamental form
$$
II=h_{ij}dx^idx^j,
$$
where $(x^i)=(x_i)$,
$u_i=\frac{\partial u}{\partial x^i}$, $F(x)=(x,u(x))$, $F_j(x)=e_j+u_j(x)e_{n+1}$,
$$
\nu=\frac{(-Du,1)}{\sqrt{1+|Du|^2}}
$$ etc, and
$$h_{ij}=(D_{F_i(x)}\nu, F_j(x))=\frac{-u_{ij}}{\sqrt{1+|D u|^2}}.
$$
Hence,
$$
II=h_{ij}dx^idx^j=\frac{D^2u}{\sqrt{1+|D u|^2}}
$$
and for $u=u(r)$,
$$
II=\frac{u_{rr}dr^2+ru_rd\theta^2}{\sqrt{1+u_r^2}}.
$$

Then the Gauss curvature of the immersed surface can be computed by
$$
K=\det(h_{ij})/\det(g_{ij})=\det(u_{ij})/(1+\nabla u|^2)^2.
$$

\begin{Thm}\label{mali}
The Cigar can not be immersed into $R^3$
\end{Thm}

\begin{proof} Assume that we can have such an immersion into $R^3$.

Since $g=g_\Sigma$ is radially symmetric, we have $z=u(\rho)$ and
$$
g=(1+u_\rho^2)d\rho^2+\rho^2d\theta^2
$$
where $u_\rho=\frac{\partial u}{\partial \rho}$, etc. Hence, we have
\begin{equation}\label{one}
(1+u_\rho^2)d\rho^2=\frac{dr^2}{1+r^2}
\end{equation}
and
\begin{equation}\label{two}
\rho^2=\frac{r^2}{1+r^2}.
\end{equation}
By (\ref{two}) we have
$$
\rho=\frac{r}{\sqrt{1+r^2}}, \ \ \ \frac{d\rho}{dr}=\frac{1}{(\sqrt{1+r^2})^3}
$$
By (\ref{one}) we have
$$
\sqrt{1+u_\rho^2}d\rho=\frac{dr}{\sqrt{1+r^2}},
$$
which implies that
$$
\sqrt{1+u_\rho^2}\cdot\frac{1}{(\sqrt{1+r^2})^3}=\frac{1}{\sqrt{1+r^2}}.
$$
Hence we have
$$
u_\rho^2=r^2
$$
and then
$$
u_\rho u_{\rho\rho}=r\frac{dr}{d\rho}={r}{(\sqrt{1+r^2})^3}.
$$
By direct computation we know that the second fundamental form can be written as
$$
II=\frac{1}{\sqrt{1+u_\rho^2}}[u_{\rho\rho}d\rho^2+\rho u_\rho d\theta^2].
$$
This would implies that the Gauss curvature $K$ is
$$
K=\frac{u_{\rho\rho}u_{\rho}\rho}{(1+u_\rho^2)^2\rho^2}=1,
$$
which is absurd. This completes the proof of Theorem \ref{mali}. \end{proof}

\section{Higher dimensional generalization}\label{sect2}
It is quite possible to show a higher dimensional analog of the result above. Namely, we may have
\begin{Thm}\label{wu} The n-dimensional radially symmetric Bryant soliton can not be isometrically immersed into $R^{n+1}$.

\end{Thm}

Recall that the n-dimensional Bryant soliton is $(R^n, g)$ ($n\geq 2$) with its Riemannian metric \cite{Cho06}
$$
g=dr^2+w(r)^2d\theta^2,
$$
where $w(r)$ is a smooth function with $w(0)=0$, $w(r)=0(r^{1/2})$, and $d\theta^2$ is the metric on $S^{n-1}$. It is well-known that it has its positive sectional curvatures
$$k_1=-\frac{w{"}}{w}=0(r^{-2}), \ \ \ k_2=\frac{1-(w')^2}{w^2}=0(r^{-1}).
$$
where $k_1$ is the curvature for the planes tangent to the radial direction $e_1=\partial_r$ and $k_2$ is the curvature for the planes tangent to the sphere.

If the n-dimensional Bryant soliton  can be imbedded into $R^{n+1}$, then we can use H.Wu's result \cite{W} as above to have it as the graph of a strictly convex radially symmetric function $u=u(x)=u(\rho), x\in R^3, \rho=|x|$. Then we have $g=(1+u_\rho^2)d\rho^2+\rho^2d\theta^2$ and
$$
dr=\sqrt{1+u_\rho^2}d\rho, \ \ w(r)=\rho=0(r^{1/2}).
$$
Hence, $r=0(\rho^2)$ and $u_\rho=0(\rho)$. We then have $u_{\rho\rho}=0(1)$.
Recall that using the components of the second fundamental form $(h_{ij})$ and $|g|=\rho^{2(n-1)}(1+u_\rho^2)$, the radial Riemannian curvature $k_1$ can also be written
$$
k_1=\frac{R_{1212}}{|g|}=\frac{u_{\rho\rho}\rho u_\rho}{\rho^{2n-2}(1+u_\rho^2)^2}=0(\rho^{-2n})=0(r^{-n}), \ \ \
$$
We may use this to find a contradiction as above.


\begin{thebibliography}{20}


\bibitem{Cho06}
 B.Chow, P.Lu, L.Ni, \emph{Hamilton's Ricci Flow}. Science
Press. American Mathematical Society, Beijing,Providence, 2006.



\bibitem{H}
R. S. Hamilton, \emph{Formation of singularities in the Ricci flow},
Surveys in Diff. Geom. 2(1995)7-136


\bibitem{W}
H.Wu, \emph{The spherical images of convex hypersurfaces}, J. Diff. Geom. 9(1974)279-290


\end{thebibliography}
\end{document}